\documentclass[a4paper, 11pt]{article}
\usepackage{amsthm, amssymb, amsmath}
\usepackage{verbatim, float}
\usepackage{mathrsfs}
\usepackage{graphics}
\usepackage[usenames,dvipsnames]{pstricks}
\usepackage{epsfig}
\usepackage{pst-grad} 
\usepackage{pst-plot} 
\usepackage{cases}
\usepackage{algorithmic}
\usepackage{subcaption}
\usepackage{tikz,pgfplots}

\newlength\figureheight 
\newlength\figurewidth 

\newtheorem{theorem}{Theorem}[section]

\newtheorem{lemma}[theorem]{Lemma}
\newtheorem{claim}[theorem]{Claim}
\newtheorem{proposition}[theorem]{Proposition}
\newtheorem{remark}[theorem]{Remark}
\theoremstyle{definition}

\theoremstyle{remark}

\begin{document}
%
\title{Directed Intersection Representations and the Information Content of Digraphs}

\author{Xujun Liu, Roberto Machado and Olgica Milenkovic\\
University of Illinois, Urbana-Champaign\\
Urbana, Illinois 61801}
\maketitle

\begin{abstract} Consider a directed graph (digraph) in which vertices are assigned color sets, and two vertices are connected if and only if they share at least 
one color and the tail vertex has a strictly smaller color set than the head. 
We seek to determine the smallest possible size of the union of the color sets that allows for such a digraph representation. 
To address this problem, we introduce the new notion of a directed intersection representation of a digraph, 
and show that it is well-defined for all directed acyclic graphs (DAGs). We then proceed to introduce the directed intersection number (DIN), the smallest number of colors needed to 
represent a DAG. Our main results are upper bounds on the DIN of DAGs based on what we call the longest terminal path decomposition of the vertex set, and constructive lower bounds.   
\end{abstract}

\section{Introduction}

In the WWW network, a number of pages are devoted to \emph{topic or item disambiguation}; in disambiguation pages, a number of identical names of designators are used to describe different entities which are further clarified and narrowed down in context via links to more specific pages. 
For example, typing the word ``Michael Jordan'' into a search engine such as Google produces a Wikipedia page which lists sportists, actors, scientists and other persons bearing this name. From this web page, one can choose to follow a link to any one of the items sharing the same two keywords, ``Michael'' and ``Jordan''. Most of the specific pages do not link back to the disambiguation page: For example, following the link to ``Michael Jordan (footballer)'' does not allow for returning to the disambiguation page, and may hence be viewed as a directed link. Furthermore, disambiguation pages tend to have little content, usually in the form of lists, while the pages that link to it tend to have significantly more information about one  of the individuals. 

Motivated by such directed networks of webpages, we consider the following problem, illustrated by a small-scale directed graph depicted in Figure~\ref{fig:web}. Assume that the vertices $A,B,C,D$ correspond to four web-pages that contain different collections of topics, files or networks, represented by color-coded rectangles (For example, each color may correspond to a different person bearing the same name). 
Two web-pages are linked to each other if they have at least one topic in common (e.g., the same name or some other shared feature). For a directed graph, in addition to the shared content assumption one needs to provide an explanation for the direction of the links, i.e., which vertex in the arc 
represents the tail and which vertex in the arc represents the head. In the context of the above described web-page linkages, it is reasonable to assume that a webpage links to another terminal webpage if the latter covers more topics, i.e., contains additional information compared to the source page. In Figure~\ref{fig:web}, the link between web-pages $A$ and $B$ is directed from $A$ to $B$, since $B$ lists three topics, while $A$ lists only two. This give rise to two generative constraints for the existence of a directed edge: Shared information content and content size dominance. This is a natural generative assumption, which has been exploited in a similar form in a number of data mining contexts~\cite{TsourakakisWWW2015,dau2017latent}.

Often, one is only presented with the directed graph topology of a directed graphs and asked to determine the latent vertex content leading to the observed topology. A problem of particular interest is to determine the \emph{smallest} topic/information content that explains the observed digraph. This question may be formally described as follows. Let $D=(V,A)$ be a directed graph with vertex set $V$ and arc set $A$, and assume that each vertex $v \in V$ is associated with a nonempty subset $\varphi(v)$ of a finite ground set $\mathcal{C}$, called the \emph{color set}, such that $(u,v) \in A$ if and only if $|\varphi(u) \cap \varphi(v)| \geq 1$ and $|\varphi(u)| < |\varphi(v)|$ (i.e., two vertices share an arc if their color sets intersect and the color set of the tail is strictly smaller than the color set of the head). If such a representation is possible, we refer to it as a \emph{directed intersection representation.} The question of interest is to determine the smallest cardinality of the ground set $\mathcal{C}$ which allows for a directed intersection representation of a digraph $D$ with $|V|=n$ vertices, henceforth termed the directed intersection number of $D$. Clearly, not all digraphs allow for such a representation. For example, a directed triangle $D\left(V, A\right)$ with $V=\{{1,2,3\}}$ and $A=\{(1,2), (2,3), (3,1)\}$ does not admit a directed intersecting representation, as such a representation would require $|\varphi(1)|< |\varphi(2)|< |\varphi(3)|< |\varphi(1)|$, which is impossible. The same is true of every digraph that contains cycles, but as we subsequently show, every directed acyclic graph (DAG) admits a directed intersection representation. We focus on connected DAGs, although our results apply to disconnected graphs with either no or some small modifications. 

\begin{figure}[!t]
\centering
\includegraphics[width=2.1in]{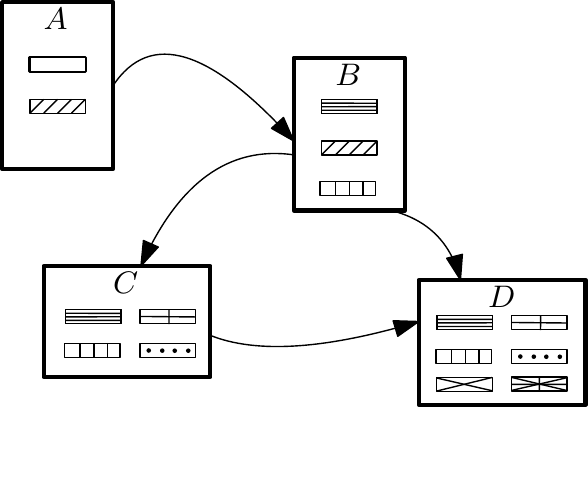}
\vspace{-0.15in}
\caption{An information storage network such as the World Wide Web. Each vertex contains a list of color-coded topics or files, representing its information content (e.g., vertex B contains a green, purple and orange topic). Vertices $A$ and $B$ are connected through an arc ($A$,$B$) since they share the green-colored topic and $A$ lists two, while $B$ lists three files.}
\label{fig:web}
\vspace{-0.2in}
\end{figure}
The problem of finding directed intersection representations of digraphs is closely associated with the intersection representation problem for undirected graphs. Intersection 
representations are of interest in many applications such as keyword conflict resolution, traffic phasing, latent feature discovery and competition graph analysis~\cite{Pullman1983, Roberts1985, dau2017triangle}. Formally, the vertices $v\in V$ of a graph $G(V,E)$ are associated with subsets $\varphi(v)$ of a ground set $\mathcal{C}$ so that $(u,v) \in E$ if and only if $|\varphi(u) \cap \varphi(v)| \geq 1$. The intersection number (IN) of the graph $G=(V,E)$ is the smallest size of the ground set  $\mathcal{C}$ that allows for an intersection representation, and it is well-defined for all graphs. Finding the intersection number of a graph is equivalent to finding the edge clique cover number, as proved by Erd\'os, Goodman and Posa in~\cite{ErdosGoodmanPosa1966}; determining the edge clique cover number is NP-hard, as shown by Orlin~\cite{Orlin1977}. The intersection number of an undirected graph may differ vastly from the DIN of some of its directed counterparts, whenever the latter exists. This is illustrated by two examples in Figure~2.
\begin{figure} \label{fig:in-vs-din}
\centering
\begin{subfigure}{.38\linewidth}
  \centering
  \setlength\figureheight{\linewidth} 
  \setlength\figurewidth{\linewidth}
   \includegraphics[width=\figurewidth,height=\figureheight]{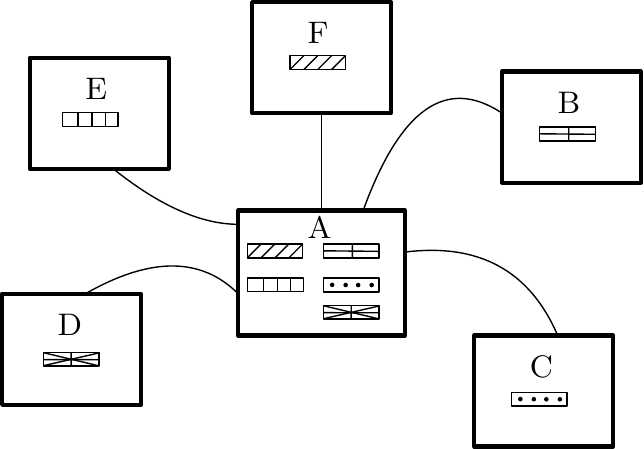}
   \caption{The intersection number of a star graph is equal to $|E|=n-1$ (e.g., $5$). }
\end{subfigure}%
\hspace{1cm}%
\begin{subfigure}{.45\linewidth}
  \centering
  \setlength\figureheight{\linewidth}
  \setlength\figurewidth{\linewidth}
  \includegraphics[width=\figurewidth,height=\figureheight]{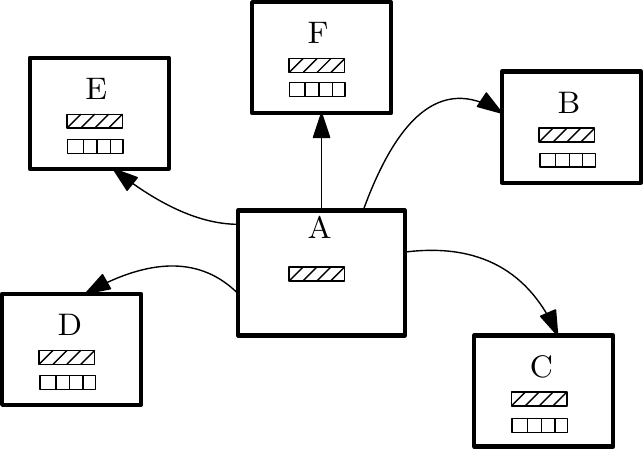}
   \caption{The DIN of any star digraph is $2$.}
\end{subfigure}\\
\vspace{0.5cm}
\begin{subfigure}{4.95cm}
  \centering
  \setlength\figureheight{4.1cm}
  \setlength\figurewidth{4.55cm}
  \includegraphics[width=\figurewidth,height=\figureheight]{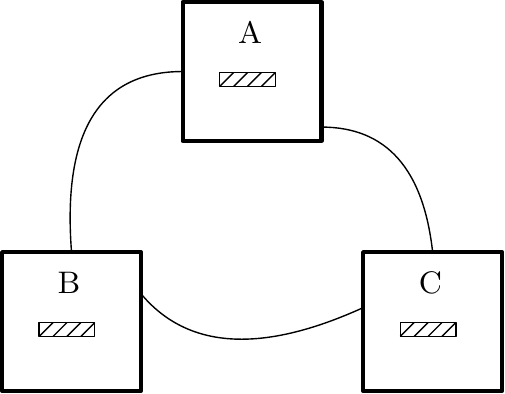}
   \caption{The intersection number of a complete graph is $1$.}
\end{subfigure} \hspace{1cm}
\begin{subfigure}{4.25cm}
  \centering
  \setlength\figureheight{4cm}
  \setlength\figurewidth{4.25cm}
  \includegraphics[width=\figurewidth,height=\figureheight]{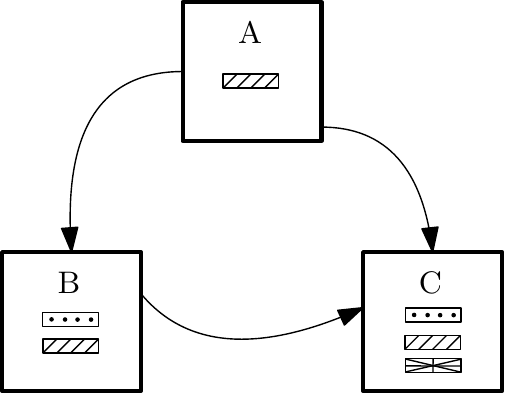}
  \caption{The DIN of a ``complete'' DAG on three vertices is exactly $3$.}
\end{subfigure}
\caption{A comparison of the intersection numbers and DINs of the star and complete graph/DAG.}
\vspace{-0.3in}
\end{figure}

The paper is organized as follows. Section~\ref{sec:families} contains a constructive proof that all DAGs have a finite directed intersection representation and algorithmically identifies representations using a suboptimal number of colors. As a consequence, the constructive algorithm establishes a bound on the DIN of arbitrary DAGs with a prescribed number of vertices. In the same section, we inductively prove an improved upper bound which is $\frac{5\,n^2}{8}-\frac{3\,n}{4}+1$. In Section~\ref{sec:extremal} we introduce the notion of DIN-extremal DAGs and describe constructions of acyclic digraphs with DINs equal to $\frac{n^2}{2} + \lfloor \frac{n^2}{16} - \frac{n}{4} + \frac{1}{4} \rfloor - 1$. 

\section{Representations of Directed Acyclic Graphs} \label{sec:families}

We use the notation and terminology described below. Whenever clear from the context, we omit the argument $n$. 

The in-degree of a vertex $v$ is the number of arcs for which $v$ is the head, while the out-degree is the number of arcs for which $v$ is the tail. 
The set of in-neighbors of $v$ is the set of vertices sharing an arc with $v$ as the head, and is denoted by $N^{-}(v)$. The set of out-neighbors $N^{+}(v)$ is defined similarly.

For a given acyclic digraph $D(V,A)$, let $\Gamma: V \rightarrow \mathbb{N}$ be a mapping that assigns to each vertex $v \in V$ 
the length of the longest directed path that terminates at $v$. The map $\Gamma$ induces a partition of the vertex 
set $V$ into \emph{levels} $(V_0, \ldots, V_{\ell}),$ such that $V_i=\{v\in V: \Gamma(v)=i\}$. We refer to $V_i,\,i=1,\ldots,\ell$ as the 
\emph{longest path decomposition} of $V$ and the graph $G$. Clearly, there is no arc between 
any pair of vertices $u$ and $v$ at the same level $V_i$, $i=1,\ldots,\ell,$ as this would violate the longest path partitioning assumption. Note that although the longest path problem is NP-hard for general graphs, it is linear time for DAGs. Finding the longest path in this case can be accomplished via topological sorting~\cite{di1988algorithms}.

\begin{lemma} \label{lemma:existence}
Every DAG $D(V,A)$ on $n$ vertices admits a directed intersection representation. 
Moreover, $DIN(n) \le \frac{5}{8}n^2-\frac{1}{4}n$. 
\end{lemma}

\begin{proof}
We prove the existence claim and upper bound by describing a constructive color assignment algorithm. 

\textbf{Step 1:} We order the vertices of the digraph as $V = (v_1,v_2, \ldots, v_n)$ so that if $(v_i,v_j) \in A,$ then $i<j$. 
One such possible ordering is henceforth referred to as a left-to-right order, and it clearly well-defined as the digraph is acyclic. 
We then construct the longest path decomposition and order the vertices in the graph starting from the first level and proceeding to the last level. 
The order of vertices inside each level is irrelevant.

\textbf{Step 2:} We group vertices into pairs in order of their labels, i.e., $(v_{2i-1},v_{2i})$, for $1 \le i \le \frac{n}{2}$, and then assign to 
each vertex $v_i,\, i=1,\ldots,n,$ a color set distinct from the color set of all other vertices. 
The sizes of the color sets equal $\frac{n}{2}-\lceil \frac{i}{2} \rceil$. 
\begin{remark}
In this step we used exactly 
\begin{equation}\label{initial}
2 \cdot \left(\frac{n}{2}-1 + \frac{n}{2}-2 + \ldots + 1\right) = 2 \cdot \frac{1+\frac{n}{2}-1}{2} \cdot \left(\frac{n}{2}-1\right) = \frac{n^2}{4} - \frac{n}{2}
\end{equation}
distinct colors. Those colors are going to be reused to accomodate for arcs between pairs.
\end{remark}

\textbf{Step 3:} For each $1 \le i \le n-2$, we assign common colors for arcs from $v_i$ to vertices belonging to pairs that follow the pair in which $v_i$ lies. More precisely:

$\bullet$ If $(v_{i},v_{2j-1}) \notin A$ and $(v_i,v_{2j}) \notin A$ for some $j$ such that $2 \cdot \lceil \frac{i}{2} \rceil < 2j-1 \le n-1 $, then we do nothing and move to the next step.

$\bullet$ If $(v_{i},v_{2j-1}) \in A$ and $(v_i,v_{2j}) \notin A$ for some $j$ such that $2 \cdot \lceil \frac{i}{2} \rceil < 2j-1 \le n-1 $, then we copy one color from $\varphi(v_i)$ not previously used in Step 3 and place it into the color set of $v_{2j-1}$, $\varphi(v_{2j-1})$.

$\bullet$ If $(v_{i},v_{2j-1}) \notin A$ and $(v_i,v_{2j}) \in A$ for some $j$ such that $2 \cdot \lceil \frac{i}{2} \rceil < 2j-1 \le n-1 $, then we copy one color from $\varphi(v_i)$ not previously used in Step 3 and place it into the color set of $v_{2j}$, $\varphi(v_{2j})$.

$\bullet$ If $(v_{i},v_{2j-1}) \in A$ and $(v_i,v_{2j}) \in A$ for some $j$ such that $2 \cdot \lceil \frac{i}{2} \rceil < 2j-1 \le n-1 $, then we copy one color from $\varphi(v_i)$ not previously used in Step 3 and place it into both $\varphi(v_{2j-1})$ and $\varphi(v_{2j})$.
\begin{remark}
Since each vertex $v_i$ has a color set $\varphi(v_i)$ with $\frac{n}{2} - \lceil \frac{i}{2} \rceil$ colors, 
and there are $\frac{n}{2} - \lceil \frac{i}{2} \rceil$ pairs following the pair that vertex $v_i$ is located in the previously fixed left-to-right ordering, 
we will never run out of colors during the above color assignment process.
\end{remark}
The color sets obtained after the previously described procedure are denoted by $\varphi'$.

\textbf{Step 4:} To the color sets of each pair of vertices $(v_{2i-1},v_{2i})$, we add at most $3i$ new colors. The augmented color sets, denoted by $\varphi''$, satisfy 1) if $v_{2i-1}v_{2i}$ is an arc, then $|\varphi''(v_{2i-1})| = \frac{n}{2}+2i-2$ and $|\varphi''(v_{2i})| = \frac{n}{2}+2i-1$; 2) if $v_{2i-1}v_{2i}$ is not an arc, then $|\varphi''(v_{2i-1})| = |\varphi''(v_{2i})| = \frac{n}{2}+2i-1$.
 
In \textbf{Step~4} we add at most 
$$\frac{n}{2}+2i-1 - 1 - \left(\frac{n}{2} - i \right) = 3i-2$$ 
colors to the color set of $v_{2i-1}$ and at 
most $$\frac{n}{2}+2i - 1 - \left(\frac{n}{2} - i\right) = 3i-1$$ 
colors to the color set of $v_{2i}$ to reach the desired color-set sizes. Note that some colors may be reused so that at this step, at most $3i-1$ new colors are 
actually needed for a pair $(v_{2i-1},v_{2i})$. Note that in \textbf{Step~3}, for each pair $(v_{2i-1},v_{2i}),$ we added in total at most $2i-2$ colors to both $\varphi'(v_{2i-1})$ and $\varphi'(v_{2i})$. Since $3i-2 > 2i-2$, we added at least one color in common for the pair 
$(v_{2i-1},v_{2i})$ so that the intersection condition is satisfied when $v_{2i-1} v_{2i}$ is an arc.

Thus, the number of colors used so far is at most 
$$
\left(3 \cdot 1 - 1\right)+\left(3 \cdot 2 - 1\right)+\ldots+\left(3\cdot \frac{n}{2}-1\right)= 3 \cdot \left(1+2+\ldots+\frac{n}{2}\right)-\frac{n}{2} 
$$
\begin{equation}\label{later}
= 3 \cdot \frac{1+\frac{n}{2}}{2} \cdot \frac{n}{2} - \frac{n}{2} = \frac{3}{8}\,n^2 + \frac{n}{4}.
\end{equation}

Next, we claim that $\varphi''$ is a valid representation that uses at most $\frac{5}{8}\, n^2 - \frac{n}{4}$ colors. From~\eqref{initial} and~\eqref{later}, we know that we used at most 
$$\frac{n^2}{4}-\frac{n}{2}+\frac{3}{8}\,n^2 + \frac{n}{4} = \frac{5}{8}\,n^2 - \frac{n}{4}$$ 
colors. 

The size condition obviously holds since $|\varphi''(v_i)| = \frac{n}{2}+i-1$ and $(v_i,v_j \in A$ implies $|\varphi(v_i)|<|\varphi(v_j)|$. The intersection condition also holds since for each $(v_i,v_j)$ with $i<j$, one has

$\bullet$ If $(v_i,v_j) \in A,$ then

1) If $(v_i,v_j)$ is a pair, then $\varphi''(v_i)$ and $\varphi''(v_j)$ have by the previous procedure at least one color in common.

2) If $(v_i,v_j)$ is not a pair, then we added a color for this arc in \textbf{Step 3}.

$\bullet$ If $(v_i,v_j) \notin A,$ then

1) If $(v_i,v_j)$ is a pair, then by previous procedure $|\varphi''(v_i)| = |\varphi''(v_j)|$. 

2) If $(v_i,v_j)$ is not a pair, then $\varphi''(v_i)$ and $\varphi''(v_j)$ have no color in common based on \textbf{Step 2} and \textbf{Step~3}.
\end{proof}

\begin{figure} 
\centering
   \includegraphics[width=5.5in]{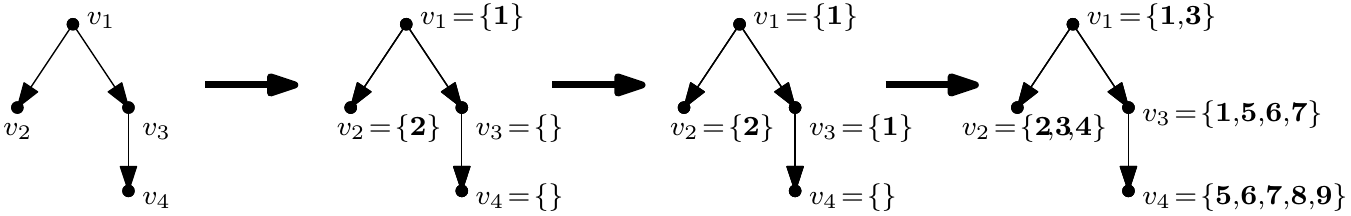}
\end{figure}

\begin{figure} 
  \centering
  \includegraphics[width=5.5in]{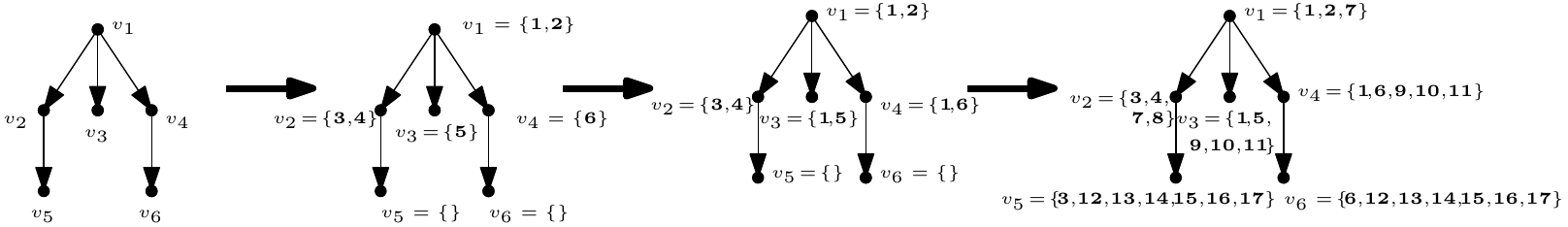}
   \caption{Directed intersection representations for two rooted trees with four and six vertices, respectively. The representations were obtained by using a vertex partition according to the longest terminal path and the constructive algorithm of Lemma~\ref{lemma:existence}.} \label{fig:comparisonalg}
\end{figure}

On the example of the directed rooted tree shown in Figure~\ref{fig:comparisonalg}, we see that more careful book-keeping and repeating of the colors used at the different levels allows one to reduce the cardinality of the representation set $\mathcal{C}$ compared to the one guaranteed by the construction of Lemma~\ref{lemma:existence}. If the vertices of the tree on the top figure are labeled according to 
the preorder traversal of the tree~\cite{treetraversal} as $v_1,v_2,v_3,$ and $v_4$, the longest terminal path vertex partition equals $V_0=\{{v_1\}}, \, V_1=\{{v_2,v_3\}}, \, V_2=\{{v_4\}}$. Using this decomposition and Lemma~\ref{lemma:existence}, we arrive at a bound for the DIN equal to $9$. It is straightforward to see the actual DIN of the tree equals $5$. Similarly, the algorithm of Lemma~\ref{lemma:existence} assigns $17$ distinct colors to the vertices of the tree depicted at the bottom of the figure, while the actual DIN of the tree equals $6$. Nevertheless, as we will see in the next section, a color assignment akin to the one described in Lemma~\ref{lemma:existence} is needed to handle a number of Hamiltonian DAGs.  

The algorithm described in the proof of Lemma~\ref{lemma:existence} established that every DAG has a directed intersection representation and introduced an algorithmic upper bound on the DIN number of any DAG on $n$ vertices with a leading term $\frac{5}{8}\,n^2$.  An improved upper bound may be obtained using (nonconstructive) inductive arguments, as described in our main result, Theorem~\ref{thm:main}, and its proof. For simplicity, we only present the proof for even $n$.

\begin{theorem} \label{thm:main}
Let $D = (V,A)$ be an acyclic digraph on $n$ vertices. If $n$ is even, then $$DIN(D) \le \frac{5n^2}{8}-\frac{3n}{4} + 1.$$
\end{theorem}

\begin{proof} 
We prove a stronger statement which asserts that for a left-to-right ordering of the vertices $V$ of an arbitrary acyclic digraph $D$, there exists a representation $\varphi$ such that 

\textbf{(a)} $|\varphi(v_{1})|=\frac{n}{2}$, $|\varphi(v_2)| \ge \frac{n}{2}$, and $|\varphi(v_i)| \ge \frac{n}{2}+1$ for $3 \le i \le n$.

\textbf{(b)} For each pair $(v_{2i-1},v_{2i})$, if $(v_{2i-1},v_{2i}) \in A$ then $|\varphi(v_{2i-1})| = |\varphi(v_{2i})| - 1$, and if $(v_{2i-1},v_{2i}) \notin A$ then $|\varphi(v_{2i-1})| = |\varphi(v_{2i})|$ for $1 \le i \le \frac{n}{2}$.

\textbf{(c)} $\cup_{i=1}^n\, \varphi(v_i)$ contains at most $\frac{5n^2}{8} - \frac{3n}{4} + 1$ colors.

The base case $n=2$ is straightforward, as a connected DAG contains only one arc. In this case, we use $\{{1\}}$ for the head and $\{1,2\}$ for the tail, and this representation clearly satisfies (a), (b), and (c).

We hence assume $n \ge 4$ and delete the arc $(v_1,v_2)$ from $D$ to obtain a new digraph $D'$; the ordering $(v_3, \ldots, v_n)$ is still a left-to-right ordering of $D'$. Thus, by the induction hypothesis, $D'$ has a representation $\varphi'$ satisfying 

\textbf{1)} $|\varphi'(v_3)|=  \frac{n}{2}-1$, $|\varphi'(v_4)| \ge \frac{n}{2}-1$, and $|\varphi'(v_i)| \ge \frac{n}{2}$ for $5 \le i \le n$; 

\textbf{2)} For each pair of vertices $(v_{2i-1},v_{2i})$, if $(v_{2i-1},v_{2i}) \in A,$ then $|\varphi(v_{2i-1})| = |\varphi(v_{2i})| - 1$, and if $(v_{2i-1},v_{2i}) \notin A,$ then $|\varphi(v_{2i-1})| = |\varphi(v_{2i})|$ for $2 \le i \le \frac{n}{2},$ and 

\textbf{3)} The representation $\varphi'$ uses at most 
\begin{equation}\label{2.5n-4}
\frac{5(n-2)^2}{8} - \frac{3(n-2)}{4} + 1 = \frac{5n^2}{8} - \frac{3n}{4} + 1 - \left(\frac{5}{2}\,n - 4\right)
\end{equation}
colors. 

We initialize our procedure by letting $\varphi = \varphi'$.

\textbf{Case 1:} $(v_1,v_2) \notin A$.

\textbf{Step 1:} Assign to $v_1$ a set of $\frac{n}{2}-1$ new colors, say $\{\alpha_1, \ldots, \alpha_{\frac{n}{2}-1}\}$. Let $\varphi(v_1) = \{\alpha_1, \ldots, \alpha_{\frac{n}{2}-1}\}$. Assign to $v_2$ a set of $\frac{n}{2}-1$ new colors, say $\{\beta_1, \ldots, \beta_{\frac{n}{2}-1}\}$, all of which are distinct from the colors in $\{\alpha_1, \ldots, \alpha_{\frac{n}{2}-1}\}$. Let $\varphi(v_2) = \{\beta_1, \ldots, \beta_{\frac{n}{2}-1}\}$.

\textbf{Step 2:} Add the same color $\gamma$ to both $\varphi(v_1)$ and $\varphi(v_2)$.

 
\textbf{Step 3:} For arcs including $v_1$, and for each $2 \le i \le \frac{n}{2},$ we perform the following procedure:
 
$\bullet$ If $(v_1,v_{2i-1}) \in A$ and $(v_1,v_{2i}) \in A$, then we copy a color from $\varphi(v_1)$ (say, $\alpha_{i-1}$) to both $\varphi(v_{2i-1})$ and $\varphi(v_{2i})$.

$\bullet$ If $(v_1,v_{2i-1}) \in A$ and $(v_1,v_{2i}) \notin A$, then we copy a color from $\varphi(v_1)$ (say, $\alpha_{i-1}$) to $\varphi(v_{2i-1})$.

$\bullet$ If $(v_1,v_{2i-1}) \notin A$ and $(v_1,v_{2i}) \in A$, then we copy a color from $\varphi(v_1)$ (say, $\alpha_{i-1}$) to $\varphi(v_{2i})$.

$\bullet$ If $(v_1,v_{2i-1}) \notin A$ and $(v_1,v_{2i}) \notin A$, then we do nothing.

\textbf{Step 4:} For arcs including $v_2$, and for each $2 \le i \le \frac{n}{2},$ we perform the following procedure:
 
$\bullet$ If $(v_2,v_{2i-1}) \in A$ and $(v_2,v_{2i}) \in A$, then we copy a color from $\varphi(v_2)$ (say, $\beta_{i-1}$) to both $\varphi(v_{2i-1})$ and $\varphi(v_{2i})$.

$\bullet$ If $(v_2,v_{2i-1}) \in A$ and $(v_2,v_{2i}) \notin A$, then we copy a color from $\varphi(v_2)$ (say, $\beta_{i-1}$) to $\varphi(v_{2i-1})$.

$\bullet$ If $(v_2,v_{2i-1}) \notin A$ and $(v_2,v_{2i}) \in A$, then we copy a color from $\varphi(v_2)$ (say, $\beta_{i-1}$) to $\varphi(v_{2i})$.

$\bullet$ If $(v_2,v_{2i-1}) \notin A$ and $(v_2,v_{2i}) \notin A$, then we do nothing.

Next, assume that the DAG representation $\varphi$ is as constructed above.

\textbf{Step 5:} For each $2 \le i \le \frac{n}{2}$, we add colors to both $\varphi(v_{2i-1})$ and $\varphi(v_{2i})$ so that the new representation $\varphi$ satisfies 
$$|\varphi(v_j)|-|\varphi'(v_j)| = 3.$$
In the process, we reuse colors to minimize the number of newly added colors. Since the procedures in Step 3 and Step 4 increase the color set of each vertex by at most $2$, one may need to add as many as $3$ new colors to a vertex representation (Note that we actually only need the difference to be $2$, but for consistency with respect to Case 2 we set the value to $3$). As an example, assume that we added $j \in \{0,1,2\}$ colors to $\varphi(v_{2i-1})$ and $k \in \{0,1,2\}$ colors to $\varphi(v_{2i})$ in Step 3 and Step 4. Then, we need to add $\max\,\{3-j,3-k\}$ colors to obtain the desired representation, which for $j=0$ or $k=0$ results in $3$ new colors. This is repeated for each pair, with at most $3$ distinct added colors.

\begin{claim}\label{color1}
The representation $\varphi$ includes at most $\frac{5}{2}\,n - 4$ new colors.
\end{claim}

\begin{proof}
We used $$\frac{n}{2}-1+\frac{n}{2}-1+1 = n - 1$$ colors in Step 1 and Step 2. We used at most $3 \cdot (\frac{n}{2}-1)$ in Step 5. Therefore, we used at most $$n - 1 + \frac{3}{2}\, n - 3 = \frac{5}{2}n - 4$$ new colors in total.
\end{proof}

\begin{claim}
The color assignments $\varphi$ constitute a valid representation satisfying conditions \textbf{(a)}, \textbf{(b)}, and \textbf{(c)}.
\end{claim}

\begin{proof}
\textbf{(i):} For a pair of vertices $(u,w)$ such that $u \in V-\{v_1,v_2\}$ and $w \in V-\{v_1,v_2\}$, we consider the following cases 

1) If $(u,w) \in A,$ then since $\varphi'$ constituted a valid representation, we have that a) the intersection condition holds for $\varphi$ because the two vertices still have representations with a color in common, and b) the size condition holds since we added three colors to both the color sets of $u$ and $w$. 

2) If $(u,w) \notin A,$ and if $u$, $w$ belong to different pairs, then since $\varphi'$ is a valid representation and we added distinct colors to different pairs of vertices in Step 5, $\varphi$ is a valid representation. This claim holds since if the vertices $u$ and $w$ have no color in common in $\varphi'$, then they still have no color in common after different colors are added in Step 5. Furthermore, if the representation sets of the vertices had the same size before we added three colors to each color set, the sizes will remain the same. If $u$, $w$ belong to the same pair, their color set sizes were the same in $\varphi'$ and they stay the same after colors are added in Step 5. Hence, $\varphi$ is still valid.

Similarly, for a pair of vertices $(u,w)$ such that $u \in \{v_1,v_2\}$ and $w \in V-\{v_1,v_2\}$, we consider the following cases. 

1) If $(u,w) \in A$, then the intersection condition holds for $\varphi$ because we added a common color to the color sets of $u$ and $w$ in Step 3 or Step 4. Furthermore, the size condition holds since $$|\varphi(w)| = |\varphi'(w)|+3 \ge \frac{n}{2}-1+3 > \frac{n}{2} = |\varphi(u)|.$$ Therefore, $\varphi$ is a valid representation.

2) If $(u,w) \notin A,$ then $\varphi$ is valid since we did not add any common color to the color sets of the two vertices, and the set $\varphi'(u)$ was obtained by augmenting it with distinct colors.

Recall that under Case 1, $(v_1,v_2) \notin A$ and $|\varphi(v_1)|=|\varphi(v_2)|$. Hence, $\varphi$ is a valid representation.

In addition, we have

\textbf{(a):} $|\varphi(v_1)|=|\varphi(v_2)|=\frac{n}{2}$ and $|\varphi(v_i)| \ge \frac{n}{2}-1+3 \ge \frac{n}{2}+1,$ for $3 \le i \le n$.

\textbf{(b):} For each pair $(v_{2i-1},v_{2i})$, if $(v_{2i-1},v_{2i}) \in A,$ then $|\varphi'(v_{2i-1})| = |\varphi'(v_{2i})| - 1$. Thus, 
$$|\varphi(v_{2i-1})| = |\varphi'(v_{2i-1})| + 3 = |\varphi'(v_{2i})| - 1 + 3 = |\varphi(v_{2i})|-1.$$
If $(v_{2i-1},v_{2i}) \notin A$, where $2 \le i \le \frac{n}{2}$, then $|\varphi'(v_{2i-1})| = |\varphi'(v_{2i})|$. Thus, 
$$|\varphi(v_{2i-1})| = |\varphi'(v_{2i-1})|+3 = |\varphi'(v_{2i})|+3 = |\varphi(v_{2i})|.$$ 
These properties also hold for $i=1$, as previously established.

\textbf{(c):} By Claim~\ref{color1}, we used at most $\frac{5}{2}n-4$ new colors.
\end{proof}

\textbf{Case 2:} $(v_1,v_2) \in A$.

\textbf{Step 1:} This step follows along the same lines as Step 1 of Case 1.

\textbf{Step 2:} Add a common color $\gamma$ to both $\varphi(v_1)$ and $\varphi(v_2)$ to satisfy the intersection constraint, 
and add a new color $\delta$ to $\varphi(v_2)$ to satisfy the size constraint. 


\textbf{Step 3:} This step follows along the same lines as Step 3 of Case 1. 

\textbf{Step 4:} This step follows along the same lines as Step 4 of Case 1.

\textbf{Step 5:} This step follows along the same lines as Step 4 of Case 1.

Using the same counting arguments as before, it can be shown that the above steps introduce $\frac{5}{2}\,n-3$ new colors (see the claim below).
\begin{claim}\label{2.5n-3}
We used at most $2.5n-3$ new colors.
\end{claim}

\begin{claim}\label{one more} One can remove (save) one color from the given representation.
\end{claim}
\begin{proof}
\textbf{Case 1:} $(v_2,v_3) \in A$.

\textbf{Case 1.1:} $(v_2,v_4) \in A$. Then $\beta_1 \in \varphi(v_3) \cap \varphi(v_4)$ and we can save one color for the pair 
$(v_3,v_4)$ in Step 5 as only two colors suffice. 

\textbf{Case 1.2:} $(v_2,v_4) \notin A$. 

\textbf{Case 1.2.1:} $(v_1,v_3) \in A$. If $(v_1,v_4) \in A,$ then $\alpha_1 \in \varphi(v_3) \cap \varphi(v_4)$ and we can save one 
color introduced in Step 5. If $(v_1,v_4) \notin A,$ then $\beta_1 \in \varphi(v_3)$ and 
$\alpha_1 \in \varphi(v_3)$. We replace $\beta_1 \in \varphi(v_3)$ by $\delta$ and replace $\beta_1 \in \varphi(v_2)$ by $\alpha_1$ and remove $\beta_1$. This saves one color.

\textbf{Case 1.2.2:} $(v_1,v_4) \in A$. Since $\beta_1 \in \varphi(v_3)$ and $\alpha_1 \in \varphi(v_4)$, we can discard one color used in Step 5. 

\textbf{Case 1.2.3:} $(v_1,v_3) \notin A$ and $(v_1,v_4) \notin A$. Then $\alpha_1$ is unused and we can thus replace $\alpha_1$ in $\varphi(v_1)$ by $\delta$ to save one color.

\textbf{Case 2:} $(v_2,v_3) \notin A$.

\textbf{Case 2.1:} $(v_2,v_4) \in A$. Then $\beta_1 \in \varphi(v_4)$. If $(v_1,v_3) \in A$, then $\alpha_1 \in \varphi(v_3)$ and 
we can save a color in Step 5. Thus, we may assume that $(v_1,v_3) \notin A$. In this case, if $(v_1,v_4) \in A$, then 
$\alpha_1 \in \varphi(v_4)$ and we replace $\alpha_1 \in \varphi(v_4)$ by a color we used in Step 5 for $v_3$ (recall that in Step 5, we added 
three new colors to $\varphi(v_3)$ and only reused one of them in $\varphi(v_4)$; hence, there are two colors remaining). In addition, we 
replace $\alpha_1 \in \varphi(v_1)$ by $\beta_1$ to save one color. Thus, we may assume $(v_1,v_4) \notin A$. Then, $\alpha_1$ is not 
used in the second pair and we may replace $\alpha_1 \in \varphi(v_1)$ by $\delta$ to save one color.

\textbf{Case 2.2:} $(v_2,v_4) \notin A$.

\textbf{Case 2.2.1:} If $(v_1,v_3) \in A$ and $(v_1,v_4) \in A$, then $\alpha_1 \in \varphi(v_3) \cap \varphi(v_4)$ and we 
saved a color in Step 5.

\textbf{Case 2.2.2:} If $(v_1,v_3) \notin A$ and $(v_1,v_4) \notin A$, then we may replace $\beta_1 \in \varphi(v_2)$ by $\alpha_1$ 
to save one color.

\textbf{Case 2.2.3:} If $(v_1,v_3) \in A$ and $(v_1,v_4) \notin A$ or $(v_1,v_3) \notin A$ and $(v_1,v_4) \in A$, then we modify Step 5 by requiring that the color sets be augmented by two rather than three colors. This allows us to save at least one color.
\end{proof}

\begin{claim}
The representation $\varphi$ is valid and it satisfies conditions \textbf{(a)}, \textbf{(b)}, and \textbf{(c)}.
\end{claim}

\begin{proof} We separately consider two cases.

$\bullet$ For Case 2.2.3,

For a pair of vertices $(u,w)$ such that $u \in V-\{v_1,v_2\}$ and $w \in V-\{v_1,v_2\}$, we consider the following cases. 

1) If $(u,w) \in A,$ then since $\varphi'$ constituted a valid representation we have that a) the intersection condition holds for $\varphi$ 
because the two vertices still have a representation with a color in common, and b) the size condition holds since we added two colors to both the color set of $u$ and $w$. 

2) If $(u,w) \notin A,$ and if $u,$ $w$ belong to different pairs, then since $\varphi'$ is a valid representation and we added 
distinct colors to different pairs in Step 5, $\varphi$ is a valid representation. This claim holds since if the vertices $u$ and $w$ have no color in common in $\varphi',$ then they still have no color in common after different colors are added in Step 5. Furthermore, if the color set representations of two vertices had the same size, then since we added two colors to both color sets, the color sets of the vertices will still have the same size. If $u,$ $w$ belong to the same pair, then their color size were the same in $\varphi'$ and remain the same after colors are added in Step 5. 
Hence, $\varphi$ is a valid representation.

Similarly, for a pair of vertices $(u,w)$ such that $u \in \{v_1,v_2\}$ and $w \in V-\{v_1,v_2,v_3,v_4\}$, we consider the following cases. 

1) If $(u,w) \in A,$ then a) the intersection condition holds for $\varphi$ because we added one common color in Step 3 or Step 4,  and b) the size condition holds since 
$$|\varphi(w)| = |\varphi'(w)|+2 \ge \frac{n}{2}+2 > \frac{n}{2}+1 \ge |\varphi(u)|.$$ 
Therefore, $\varphi$ is a valid representation.

2) If $(u,w) \notin A,$ then $\varphi$ is valid since 
$$\varphi(w) \ge \frac{n}{2}+2 > \frac{n}{2}+1 \ge \varphi(u)$$ 
and we did not add a common color for the two vertices, and $\varphi'(u)$ was obtained by adding distinct colors to $\varphi(u)$.

For $(v_1,v_3)$, when $(v_1,v_3) \in A$ we added $\alpha_1$ to $\varphi(v_3)$ so that
 $$|\varphi(v_3)| = \frac{n}{2}+1 > \frac{n}{2} = |\varphi(v_1)|.$$ When $(v_1,v_3) \notin A$ we added distinct colors to $\varphi(v_1)$ and $\varphi(v_3)$. Thus, $\varphi$ is valid.

For $(v_1,v_4)$, when $(v_1,v_4) \in A$ we added $\alpha_1$ to $\varphi(v_4)$ so that
 $$|\varphi(v_4)| = \frac{n}{2}+1 > \frac{n}{2} = |\varphi(v_1)|.$$ When $v_1v_4 \notin A$ we added distinct colors to $\varphi(v_1)$ and $\varphi(v_4)$. Thus, $\varphi$ is valid.

For $(v_2,v_3)$, we added distinct colors to $\varphi(v_2)$ and $\varphi(v_3)$. Thus, $\varphi$ is valid.

For $(v_2,v_4)$, we added distinct colors to $\varphi(v_2)$ and $\varphi(v_4)$. Thus, $\varphi$ is valid.

For $(v_1,v_2)$, since $(v_1,v_2) \in A$, $\gamma \in \varphi(v_1) \cap \varphi(v_2)$, and $|\varphi(v_1)|=|\varphi(v_2)|-1$ we have that $\varphi$ is valid.

To verify that conditions \textbf{(a)}, \textbf{(b)} and \textbf{(c)} are satisfied, observe that:

\textbf{(a):} $|\varphi(v_1)|=|\varphi(v_2)|-1=\frac{n}{2}$ and $|\varphi(v_i)| \ge \frac{n}{2}+1$ for $3 \le i \le n$.

\textbf{(b):} For each pair $(v_{2i-1},v_{2i})$, if $(v_{2i-1},v_{2i}) \in A$ then $|\varphi'(v_{2i-1})| = |\varphi'(v_{2i})| - 1$. Thus, $$|\varphi(v_{2i-1})| = |\varphi'(v_{2i-1})| + 2 = |\varphi'(v_{2i})| - 1 + 2 = |\varphi(v_{2i})|-1.$$ This claim is also true for $i=1$, which we already showed.

If $(v_{2i-1},v_{2i}) \notin A$, where $2 \le i \le \frac{n}{2}$, then $|\varphi'(v_{2i-1})| = |\varphi'(v_{2i})|$. Thus, $$|\varphi(v_{2i-1})| = |\varphi'(v_{2i-1})|+2 = |\varphi'(v_{2i})|+2 = |\varphi(v_{2i})|.$$ 

\textbf{(c):} By Claim~\ref{2.5n-3} and Claim~\ref{one more}, we used at most $2.5n-4$ new colors.

$\bullet$ For the other cases, 

For a pair of vertices $(u,w)$ such that $u \in V-\{v_1,v_2\}$ and $w \in V-\{v_1,v_2\}$, we consider the following cases. 

If $(u,w) \in A$ then since $\varphi'$ was valid 1) the intersection condition still holds for $\varphi$ because they still have color in common and 2) the size condition still hold since we added three colors to each of the color set of $u$ and $w$. 

If $(u,w) \notin A$, and the two vertices are in different pairs then since $\varphi'$ was valid and we added distinct colors to different pairs in Step 5, we have that $\varphi$ is valid because if $u$ and $v$ have no color in common in $\varphi'$ then they still have no color in common after we added different colors in Step 5; if they had the same size in $\varphi'$ then since we added three colors to each color set their sizes remain the same. If the two vertices are in the same pair then their color size was the same in $\varphi'$ and it stays the same after adding colors in Step 5. Hence, $\varphi$ is still valid.

For a pair of vertices $(u,w)$ such that $u \in \{v_1,v_2\}$ and $w \in V-\{v_1,v_2\}$, we consider the following cases. 

If $(u,w) \in A$ then 1) the intersection condition holds for $\varphi$ because we added a common color in Step 3 or Step 4 to the color sets of $u$ and $w$ and 2) the size condition hold since $$|\varphi(w)| = |\varphi'(w)|+3 \ge \frac{n}{2}-1+3 > \frac{n}{2}+1 \ge |\varphi(u)|.$$ Therefore, $\varphi$ is valid.

If $(u,w) \notin A$ then $\varphi$ is valid since we did not add any common color for them and $u$ uses distinct colors from $\varphi'$.

For $(v_1,v_2)$, since $(v_1,v_2) \in A$, $\gamma \in \varphi(v_1) \cap \varphi(v_2)$, and $|\varphi(v_1)|=|\varphi(v_2)|-1$ we have that $\varphi$ is valid.

To verify that conditions \textbf{(a)}, \textbf{(b)} and \textbf{(c)} are satisfied, observe that:

\textbf{(a):} $|\varphi(v_1)|=|\varphi(v_2)|-1=\frac{n}{2}$ and $|\varphi(v_i)| \ge \frac{n}{2}-1+3 \ge \frac{n}{2}+1$ for $3 \le i \le n$.

\textbf{(b):} For each pair $(v_{2i-1},v_{2i})$, if $(v_{2i-1},v_{2i}) \in A$ then $|\varphi'(v_{2i-1})| = |\varphi'(v_{2i})| - 1$. Thus, $$|\varphi(v_{2i-1})| = |\varphi'(v_{2i-1})| + 3 = |\varphi'(v_{2i})| - 1 + 3 = |\varphi(v_{2i})|-1.$$ This claim is also true for $i=1$, which we already showed.

If $(v_{2i-1},v_{2i}) \notin A$, where $2 \le i \le \frac{n}{2}$, then $|\varphi'(v_{2i-1})| = |\varphi'(v_{2i})|$. Thus, $$|\varphi(v_{2i-1})| = |\varphi'(v_{2i-1})|+3 = |\varphi'(v_{2i})|+3 = |\varphi(v_{2i})|.$$ 

\textbf{(c):} By Claim~\ref{2.5n-3} and Claim~\ref{one more}, we used at most $2.5n-4$ new colors.\\
This proves the claim. \end{proof}
This completes the proof of the theorem. \end{proof}

\section{Extremal DIN Digraphs and Lower Bounds} \label{sec:extremal}

The derivations in the previous section proved that for any DAG $D$ on $n$ vertices, one has 
\begin{equation} \label{eq:best}
DIN(D) \le  \frac{5n^2}{8}-\frac{3n}{4} + 1.
\end{equation}

In comparison, the intersection number of any graph on $n$ 
vertices is upper bounded by $\frac{n^2}{4}$~\cite{ErdosGoodmanPosa1966}. Furthermore, the existence of undirected graphs 
that meet the bound $\frac{n^2}{4}$ can be established by observing that the intersection number of a graph is equivalent to its 
edge-clique cover number and by invoking Mantel's theorem~\cite{mantel1907problem} which asserts that any triangle-free graph on $n$ vertices
can have at most $\frac{n^2}{4}$ edges. The extremal graphs with respect to the intersection number are the well-known Turan graphs $T(n,2)$~\cite{turan1954theory}. 

Consequently, the following question is of interest in the context of directed intersection representations: Do there exist DAGs that meet the upper 
bound in~\eqref{eq:best} and which DIN values are actually achievable? To this end, we introduce the notion of \emph{DIN-extremal} DAGs: A DAG on $n$ vertices is said to be DIN-extremal if it has the largest DIN among all DAGs with the same number of vertices. 

Directed path DAGs, e.g., directed acyclic graphs $D(V,A)$ with $V=\{{1,2,\ldots,n\}}$ and $A=\{{(1,2),(2,3),(3,4),\ldots,(n-1,n)\}}$ have DINs that scale as $\frac{n^2}{4}$. The following result formalizes this observation. 
\begin{proposition}
Let $D(V,A)$ be a directed path on $n$ vertices. If $n$ is even, then $DIN(D)=\frac{n^2+2n}{4}$; if $n$ is odd, then $DIN(D)=\frac{n^2+2n+1}{4}$.
\end{proposition}
The proof of the result is straightforward and hence omitted.

Figure~4 provides examples of DIN-extremal DAGs for $n\leq 7$ vertices. 
These graphs were obtained by combining computer simulations and proof techniques used in 
establishing the upper bound of~\eqref{eq:best}. Direct verification for large $n$ through 
exhaustive search is prohibitively complex, as the number of connected/disconnected DAGs with $n$ vertices follows a ``fast growing'' recurrence~\cite{robinson1977counting}.
For example, even for $n=6$, there exist $5984$ different unlabeled DAGs. Note that all listed extremal DAGs are Hamiltonian, e.g., they contain a 
directed path visiting each of the $n$ vertices exactly once. As 
such, the digraphs have a unique topological order induced by the directed path, and for the decomposition described on page $5$ one has $|V_i|=1$ for all $i \in [n]$. Note that the bound in~\eqref{eq:best} for $n=2,3,4,5, 6,7$ equals $2,4,8,12,19,26$, respectively. Hence, the upper bound in (4) is loose for $n\geq6$.

\begin{center} \label{fig:examples}
				\begin{figure}[h]
				\hspace{0.8in}
		\includegraphics[width=8cm]{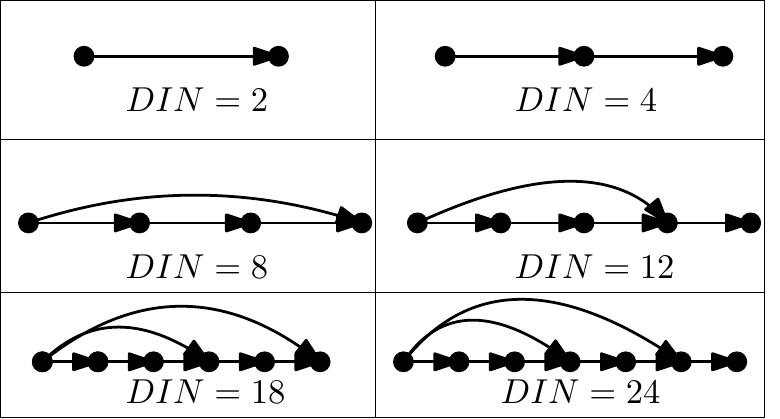}
		\hspace{0.2in}  \caption{Examples of DIN-extremal graphs for $n\leq 7$.}
				\end{figure}
	\end{center}
	
\begin{figure} \label{fig:rainbows}
\centering
\begin{subfigure}{2.3in}
  \centering
   \includegraphics[width=2.3in,height=0.5in]{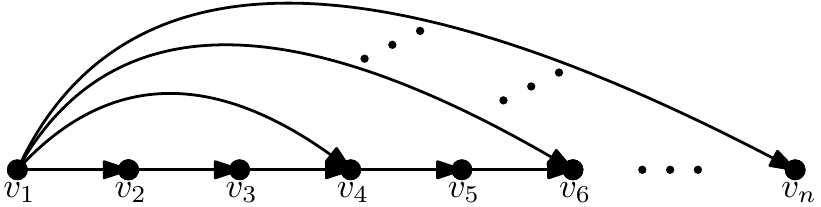}
   \caption{Source arc-path, $n$ even.}
\end{subfigure}%
\hspace{0.9cm}%
\begin{subfigure}{2.3in}
  \centering
   \includegraphics[width=2.3in,height=0.5in]{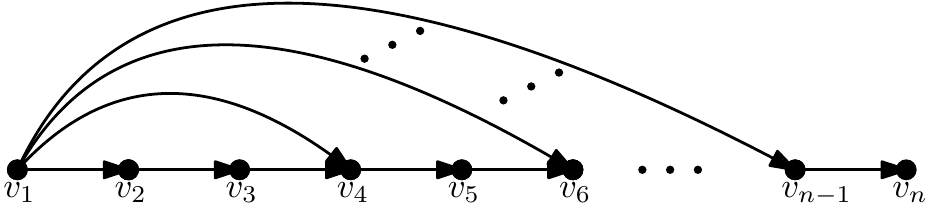}
   \caption{Source arc-path, $n$ odd.}
\end{subfigure}%
\end{figure}
	
For all $n\leq7$ the extremal digraphs are what we refer to as \emph{source arc-paths}, illustrated in Figure~5 a),b). 
A source arc-path on $n$ vertices has the following arc set  
$$A=\{(v_1, v_{2k}): k\in [\lfloor n/2 \rfloor]\}\cup \{(v_{k},v_{k+1}): k \in [n\!-\!1]\}.$$
It is straightforward to prove the following result.
\begin{proposition}
The DIN of a source arc-path on $n$ vertices is equal to $\lfloor \frac{n^2}{2} \rfloor=\lfloor \frac{4n^2}{8} \rfloor$. Hence, the DIN of source arc-paths is by $\frac{n^2}{8}$ smaller than the leading term of the upper bound~\eqref{eq:best}. 
\end{proposition}
\begin{proof}
A \emph{directed triangle} in a digraph $D = (V, A)$ is a collection of three vertices $\{{v_i,v_j,v_k\}}$ such that $(v_i,v_j) \in A$, $(v_j,v_k) \in A$, and $(v_i,v_k) \in A$. Since a source arc-path avoids directed triangles and every vertex has a color set of different size than another (due to the presence of the directed Hamiltonian path), every color may be used at most twice. We need $\frac{n}{2}$ colors for $\varphi(v_1)$ to represent the arcs $v_1v_{2i}$, where $1 \le i \le \frac{n}{2}$. Since the size of the color sets $\varphi$ increases along the directed path, vertex $v_j$ in the natural ordering has $\varphi(v_j) \geq \frac{n}{2}+j-1$. Furthermore, $(v_{2i},v_{2j}) \notin A$ for a source arc-path, for all $1 \le i < j \le \frac{n}{2}$. Thus, $\varphi(v_{2i}) \cap \varphi(v_{2j})=\emptyset$, $1 \le i < j \le \frac{n}{2}$. This implies the number of colors needed is $$ \ge \frac{n}{2}+1 + \frac{n}{2}+3 + \cdots + \frac{n}{2}+ n-1 = \frac{n}{2} \cdot \frac{n}{2} + \frac{(1+n-1)(\frac{n}{2})}{2} = \frac{n^2}{2}.$$ 

To show that the above lower bound is met, we exhibit the following representation $\varphi$ with $\frac{n}{2}$ colors:\\ 
\vspace{0.08in}
1) $\varphi(v_1)=\{c_1, \ldots, c_{\frac{n}{2}}\}$, $\varphi(v_2)= \{c_1, f_{1}, g_{1,1}, \ldots, g_{\frac{n}{2}-1,1}\}.$\\
2) For $2 \le i \le \frac{n}{2}-1$, $$\varphi(v_{2i})= \{c_i, d_{i}, f_{i}, g_{1,i}, \ldots, g_{\frac{n}{2}+2i-4,i}\},$$ 
$$\varphi(v_n)=\{c_{\frac{n}{2}}, d_{\frac{n}{2}}, g_{1,\frac{n}{2}}, \ldots, g_{\frac{n}{2}+n-3, \frac{n}{2}}\}.$$
3) For $2 \le i \le \frac{n}{2}-1$, $$\varphi(v_{2i-1}) = \{d_i, f_{i-1}, g_{1,i}, \ldots, g_{\frac{n}{2}+2i-4,i}\}.$$
$$\varphi(v_{n-1}) = \{f_{\frac{n}{2}-1}, d_{\frac{n}{2}}, g_{1, \frac{n}{2}}, \ldots, g_{\frac{n}{2}+n-4, \frac{n}{2}}\}.$$
\end{proof}

For $n\geq 8$, there exist DAGs with DINs that exceed those of source arc-paths which are obtained by adding carefully selected additional arcs. 
For even integers $n$, the DIN of such graphs equals $$\frac{n^2}{2} + \lfloor \frac{n^2}{16} - \frac{n}{4} + \frac{1}{4} \rfloor - 1.$$ 
A digraph with the above DIN has a vertex set $V = \{v_1, \ldots, v_n\}$ and arcs constructed as follows:

\textbf{Step 1:} Initialize the arc set as $A = \emptyset$.

\textbf{Step 2:} Add to $A$ arcs of a source-arc-path, i.e., 
$$A = A \cup \{(v_1,v_{2i}): i \in [\frac{n}{2}]\} \cup \{(v_j, v_{j+1}): j \in [n-1]\}.$$

\textbf{Step 3:} Add arcs with tails and heads in the set $\{v_3, v_5, \ldots, v_{n-1}\}$ according to the following rules:

\textbf{Step 3.1:} If $\frac{n-2}{2}$ is even, then let $X = \{v_3, v_5, \ldots, v_{\frac{n}{2}}\}$ and $Y = \{v_{\frac{n}{2}+2}, \ldots, v_{n-1}\}$. Add all arcs 
between $X$ and $Y$ except for $(v_{\frac{n}{2}},v_{\frac{n}{2}+2})$. 

\textbf{Step 3.2:} If $\frac{n-2}{2}$ is odd, then let $X = \{v_3, v_5, \ldots, v_{\frac{n}{2}+1}\}$ and $Y = \{v_{\frac{n}{2}+3}, \ldots, v_{n-1}\}$. Add all arcs 
between $X$ and $Y$ except for $(v_{\frac{n}{2}+1},v_{\frac{n}{2}+3})$.

The above described digraphs have no directed triangles and their number of arcs equals 
$$\lfloor \frac{(\frac{n}{2}-1)^2}{4} \rfloor - 1 = \lfloor \frac{n^2}{16} - \frac{n}{4} + \frac{1}{4} \rfloor - 1.$$

We start with the following lower bound on the DIN number of the augmented source-arc-path graphs.

\begin{proposition}
The DIN of the above family of graphs is at least 
$$\frac{n^2}{2} + \lfloor \frac{n^2}{16} - \frac{n}{4} + \frac{1}{4}  \rfloor - 1.$$
\end{proposition}

\begin{proof}
Due to the presence of the arc of a source-arc-path, $v_1$ requires at least $\frac{n}{2}$ colors. Furthermore, since the graph is Hamiltonian, the size of the color sets increases along the path. Based on the previous two observations, one can see that $v_i$ requires at least $\frac{n}{2}+i-1$ colors for all $i \in [n]$.

Since there are no arcs in the digraph induced by the vertex set $\{v_2, v_4, \ldots, v_n\}$ with even labels, the color sets of these vertices have to be mutually disjoint. Thus, the number of colors needed to color vertices with even indices is at least 

$$\frac{n}{2}+1+\frac{n}{2}+3+ \ldots + \frac{n}{2}+n-1 = \frac{n^2}{2}.$$

Since the digraphs avoid directed triangles and every pair of vertices has a different color set sizes, we require one additional color to represent each of the arcs added in Step 3. Due to the absence of directed triangle, we need at least $\lfloor \frac{n^2}{16} - \frac{n}{4} + \frac{1}{4}  \rfloor - 1$ colors. Furthermore, the color sets used for the two previously described vertex sets are disjoint. Thus, the number of colors required is at least 
$$\frac{n^2}{2} + \lfloor \frac{n^2}{16} - \frac{n}{4} + \frac{1}{4}  \rfloor - 1.$$
\end{proof}

To show that the above number of colors suffices to represent the digraphs under consideration, we provide next a representation $\varphi$ using $\frac{n^2}{2} + \lfloor \frac{n^2}{16} - \frac{n}{4} + \frac{1}{4}  \rfloor - 1$ colors.

We start by exhibiting a representation $\varphi'$ of the source-arc-path that uses $\frac{n^2}{2}$ colors and then change the color assignments accordingly:

1) Set $\varphi(v_1)=\{c_1, \ldots, c_{\frac{n}{2}}\}$ and $\varphi(v_2)= \{c_1, f_{1}, g_{1,1}, \ldots, g_{\frac{n}{2}-1,1}\}.$

2) For $2 \le i \le \frac{n}{2}-1$, set 
$$\varphi(v_{2i})= \{c_i, d_{i}, f_{i}, g_{1,i}, \ldots, g_{\frac{n}{2}+2i-4,i}\},$$ 
and
$$\varphi(v_n)=\{c_{\frac{n}{2}}, d_{\frac{n}{2}}, g_{1,\frac{n}{2}}, \ldots, g_{\frac{n}{2}+n-3, \frac{n}{2}}\}.$$

3) For $2 \le i \le \frac{n}{2}-1$, set 
$$\varphi(v_{2i-1}) = \{d_i, f_{i-1}, g_{1,i}, \ldots, g_{\frac{n}{2}+2i-4,i}\},$$
and
$$\varphi(v_{n-1}) = \{f_{\frac{n}{2}-1}, d_{\frac{n}{2}}, g_{1, \frac{n}{2}}, \ldots, g_{\frac{n}{2}+n-4, \frac{n}{2}}\}.$$

Let $m := \lfloor \frac{n^2}{16} - \frac{n}{4} + \frac{1}{4}  \rfloor - 1$.

1') Set $\Gamma_{2i-1} = \{g_{1,i}, \ldots, g_{\frac{n}{2}+2i-4,i}\}$.

2') Order the $m$ arcs in the graph induced by $\{v_3, v_5, \ldots, v_{n-1}\}$ in an arbitrary fashion, say $\{e_1, \ldots, e_m\}$. Set a counter variable to $k = 1$.

3') For $e_k = (v_{2i-1},v_{2j-1}),$ assign a previously unused color $h_k$ to both $\varphi(v_{2i-1})$ and $\varphi(v_{2j-1})$. 
Pick one color $g'$ from $\Gamma_{2i-1}$ and a color $g''$ from $\Gamma_{2j-1}$ not previously used in the procedure. Set 
$$\varphi(v_{2i-1}) = \varphi(v_{2i-1}) \cup h_k - g', \hspace{0.89cm} \text{and} \hspace{1cm}  \Gamma_{2i-1} = \Gamma_{2i-1} - g', $$
 $$\varphi(v_{2j-1}) = \varphi(v_{2j-1}) \cup h_k - g'',  \hspace{8mm} \text{and} \hspace{1cm} \Gamma_{2j-1} = \Gamma_{2j-1} - g''.$$
 Let $k = k+1$. If $k \le m$, go to Step 3'), otherwise stop.
 
 4') Since each $v_{2i-1}$ has degree at most $\frac{n}{4}$ on the digraph induced by $\{v_3, \ldots, v_{n-1}\}$ and at step $k=1$ we had $|\Gamma_{2i-1}| = \frac{n}{2}+2i-4$, we do not run out of colors to replace. This follows since when we choose $g'$ from $\Gamma_{2i-1}$ we always have $\geq \frac{n}{2}+2i-4-\frac{n}{4}$ colors available.
 
 5') Since $g',g''$ were used twice in $\varphi'$ and deleted only once in the processing steps (and thus remain in the union of the colors), each iteration of the 
 procedure in 3) introduces exactly one new color (e.g., $h_k$) to $\varphi$. Therefore, the number of colors used is 
 
 $$\frac{n^2}{2}+m = \frac{n^2}{2}+\lfloor \frac{n^2}{16} - \frac{n}{4} + \frac{1}{4}  \rfloor - 1.$$

This completes the construction of digraphs on $n$ vertices with DIN values $\frac{n^2}{2}+\lfloor \frac{n^2}{16} - \frac{n}{4} + \frac{1}{4}  \rfloor - 1.$

\section{Open Problems}

We conclude the paper by listing a number of open problems and extensions of the line work introduced in the paper.
\begin{itemize}
\item Improve the upper bound in (4) and the constructive lower bound in Proposition 3.3.
\item Prove that for each $n$, there exists a DIN-extremal digraph that is Hamiltonian. 
\item Extended the notion of directed intersection representation to include $p$-intersections, $p>1$, for which the generative size constraint equals $|\varphi(u)\cap\varphi(v)|>p$. It is straightforward to see that $DIN_p(D) \leq DIN(D)+p-1$, where $DIN_p(D)$ directs the directed $p$-intersection number. This observation follows from the observation that adding $p-1$ common colors to the vertices suffices to satisfy the required constraints. Sharper bounds are currently unknown.
\end{itemize}

\section*{Acknowledgment}
The authors gratefully acknowledge many useful discussions with Prof. Alexandr Kostochka from the University of Illinois and are indebted to him for suggesting new proof techniques. The work was supported by the NSF STC Center for Science of Information, 4101-38050, the S\~{a}o Paulo Research Foundation grant 2015/11286-8, the grant NSF CCF 15-26875, and UIUC Research Board Grant RB17164.


%
\bibliographystyle{IEEEtran}
\bibliography{intersection}

\end{document}